\documentclass[english]{article}

\usepackage[T1]{fontenc}
\usepackage[latin9]{inputenc}
\usepackage{float}
\usepackage{hyperref}
\usepackage{amsmath}
\usepackage{mathrsfs}
\usepackage{verbatim}
\usepackage{graphicx}
\usepackage{amssymb}
\usepackage{framed}
\usepackage{subfigure}
\usepackage{babel}
\usepackage[a4paper]{geometry}
\usepackage{enumerate}
\usepackage{longtable}
\usepackage{mathtools}
\mathtoolsset{showonlyrefs}
\usepackage{mdwlist}
\usepackage{paralist}
\usepackage{pgf,tikz}
\usetikzlibrary[patterns]
\usetikzlibrary{arrows}

\newtheorem{theorem}{Theorem}[section]

\newtheorem{proposition}[theorem]{Proposition}
\newtheorem{definition}[theorem]{Definition}

\newtheorem{remark}[theorem]{Remark}

\numberwithin{equation}{section}

\newcommand{\qed}{\hspace*{\fill} $\blacksquare$\medskip}

\newenvironment{proof}
{\noindent {\em Proof}.\,\,}
{\qed}

\def \Z {\mathbb Z}
\def \R {\mathbb R}

\def \N {\mathbb N}

\def \ee {\mathrm e}
\def \dd {\mathrm d}
\def \cR {\mathcal R}
\def \b {\beta}
\def \L {\Lambda}
\def \T {\Theta}
\def \h {\eta}
\def \cX {\mathcal X}
\def \cL {\mathcal L}
\def \n {\mu}
\def \D {\Delta}
\def \d {\delta}

\def \SES {{\hbox{\footnotesize\rm SES}}}

\newcommand{\dist}{\mathrm{dist}}

\begin{document}

\author{
\renewcommand{\thefootnote}{\arabic{footnote}}
S.\ Baldassarri \footnotemark[1] 
\\
\renewcommand{\thefootnote}{\arabic{footnote}}
A.\ Gaudilli\`ere \footnotemark[2]
\\
\renewcommand{\thefootnote}{\arabic{footnote}}
F.\ den Hollander \footnotemark[3]
\\
\renewcommand{\thefootnote}{\arabic{footnote}}
F.R.\ Nardi \footnotemark[4]
\\
\renewcommand{\thefootnote}{\arabic{footnote}}
E.\ Olivieri \footnotemark[5]
\\
\renewcommand{\thefootnote}{\arabic{footnote}}
E.\ Scoppola \footnotemark[6]
}

\title{Homogeneous nucleation for\\ 
two-dimensional Kawasaki dynamics}

\footnotetext[1]{
	Gran Sasso Science Institute, Viale Francesco Crispi 7, 67100 L'Aquila, Italy
}

\footnotetext[2]{
	Aix-Marseille Universit\'e, CNRS, I2M, 3 place Victor Hugo, 13\,003 Marseille, France
}

\footnotetext[3]{
	Mathematisch Instituut, Universiteit Leiden, Einsteinweg 55, 2333 CC Leiden, The Netherlands
}

\footnotetext[4]{
Universit\`a di Firenze, Dipartimento di Matematica e Informatica ``Ulisse Dini", Viale Morgagni 67/a 50134, Firenze, Italy
}

\footnotetext[5]{
Dipartimento di Matematica, Universit\`a di Roma Tor Vergata, Via della Ricerca Scientifica, 00133 Roma, Italy
}

\footnotetext[6]{
Dipartimento di Matematica e Fisica, Universit\`a di Roma Tre, Largo S.\ Leonardo Murialdo 1, 00146 Roma, Italy
}

\date{}

\maketitle


\begin{abstract}
This is the third in a series of three papers in which we study a lattice gas subject to Kawasaki dynamics at inverse temperature $\beta>0$ in a large finite box $\Lambda_\beta \subset \Z^2$ whose volume depends on $\beta$. Each pair of neighbouring particles has a \emph{negative binding energy} $-U<0$, while each particle has a \emph{positive activation energy} $\Delta>0$. The initial configuration is drawn from the grand-canonical ensemble restricted to the set of configurations where all the droplets are subcritical. Our goal is to describe, in the metastable regime $\Delta \in (U,2U)$ and in the limit as $\beta\to\infty$, how and when the system nucleates, i.e., creates a critical droplet somewhere in $\Lambda_\beta$ that subsequently grows by absorbing particles from the surrounding gas.  

In the first paper we showed that subcritical droplets behave as quasi-random walks. In the second paper we used the results in the first paper to analyse how subcritical droplets form and dissolve on multiple space-time scales when the volume is \emph{moderately large}, namely, $|\Lambda_\beta| = \ee^{\theta\beta}$ with $\Delta < \theta < 2\Delta-U$. In the present paper we consider the setting where the volume is \emph{very large}, namely, $|\Lambda_\beta| = \ee^{\Theta\beta}$ with $\Theta < \Gamma-(2\Delta-U)$, where $\Gamma$ is the energy of the critical droplet in the local model with fixed volume, and use the results in the first two papers to identify the nucleation time and the tube of typical trajectories towards nucleation. We will see that in a very large volume critical droplets appear more or less independently in boxes of moderate volume, a phenomenon referred to as \emph{homogeneous nucleation}. This leads a nucleation time of order $\ee^{\Gamma\beta}/|\Lambda_\beta|$. One of the key ingredients in the proof is an estimate showing that no information can travel between these boxes on relevant time scales. Prior to nucleation, the dynamics inside the moderately large boxes spends most of its time in configurations consisting of a single quasi-square surrounded by free particles, and makes transitions between these configurations at rates that depend on the linear sizes of the quasi-square in a computable manner.    

\newpage

\\
{\it Acknowledgement.}
FdH and FRN were supported through NWO Gravitation Grant NETWORKS 024.002.003. The work of SB was supported by the European Union's Horizon 2020 research and innovation programme under the Marie Sk\l{}odowska-Curie grant agreement no.\ 101034253 while affiliated with Leiden University. SB was further supported through LYSM and ``Gruppo Nazionale per l'Analisi Matematica, la Probabilit\'a e le loro Applicazioni'' (GNAMPA-INdAM). We dedicate our trilogy of papers to Francesca Nardi, who was part of the long adventure to understand Kawasaki dynamics in large volumes, and unfortunately did not live to see the end of her work. All three papers saw the light thanks to Francesca's unwavering commitment and fierce determination.
\end{abstract}


\small
\tableofcontents
\normalsize


\section{Introduction and main results}


\subsection{Background}

The present paper is the third in a series of three papers dealing with \emph{nucleation} in a \emph{supersaturated lattice gas} in a large volume. In particular, we consider a two-dimensional lattice gas at low density and low temperature that evolves under Kawasaki dynamics, i.e., particles hop around randomly subject to hard-core repulsion and nearest-neighbour attraction. We are interested in how the gas \emph{nucleates} in large volumes, i.e., how the particles form and dissolve subcritical droplets until they manage to build a critical droplet that is large enough to trigger the nucleation. In large volumes the evolution of the Kawasaki lattice gas is much more involved than in finite volumes treated earlier \cite{dHOS00a}, \cite{dHOS00b}, \cite{GOS04}. The main difficulty in analysing the metastable behaviour is a proper description of the interaction between the droplets and the surrounding gas. As part of the nucleation process, droplets grow and shrink by exchanging particles with the gas around them, as is typical for \emph{conservative dynamics}. 
	
In the first paper \cite{GdHNOS09} we showed that subcritical droplets behave as quasi-random walks. In the second paper \cite{BGdHNOS23} we used the results in the first paper to analyse how subcritical droplets form and dissolve \emph{on multiple space-time scales} when the volume is \emph{moderately large}. In the present paper we use the results of the first and the second paper to complete our description of the nucleation. As we will see, this requires new and delicate arguments. 
	
The present work resolves the challenges put forward in \cite{dHOS00a} for Kawasaki dynamics in large volumes at low temperature. Earlier work was carried out in \cite{BdHS10}, \cite{GL}. In \cite{BdHS10} the average nucleation time was identified, including a \emph{sharp prefactor}, but \emph{without} providing information on \emph{how} the nucleation takes place. In addition, the results in \cite{BdHS10} only hold for a specific starting distribution called \emph{the last-exit-biased distribution}, which is tuned to the potential-theoretic approach to metastability. This distribution is \emph{unphysical}, in the sense that it is unclear how to prepare a low-density gas in such a way that it starts from this distribution. In contrast, the \emph{Gibbs distribution restricted to subcritical droplets} in \eqref{muR} employed here as starting distribution is \emph{physical}, because it can be obtained by rapidly cooling down a low-density gas. With extra work the prefactor found in \cite{BdHS10} can in principle be obtained as well, since, following \cite{BGM20}, we essentially have to provide rough bounds on local relaxation times only, which can be obtained from \cite{BGdHNOS23}. Finally, in \cite{GL} the transitions between the different ground states are analysed in a regime where there is no pure-gas metastable state and the system starts from a large square droplet with no surrounding gas. In that setting the interaction between the gas and the droplet, which is at the core of the present work, is largely avoided.
	
Our focus is on both the \emph{nucleation time} and the \emph{tube of typical trajectories prior to nucleation}. To derive our main theorems we distinguish between \emph{moderate volumes} and \emph{large volumes}. We first consider moderate volumes, and thanks to the results obtained in \cite{BGdHNOS23} we are able to \emph{fully} describe the escape from metastability. In particular, we show that this escape occurs via \emph{nucleation} rather than via \emph{coalescence}. Subsequently, we consider large volumes. To get the results in this case, we need some form of independence on certain time scales between events occurring in regions separated by more than a \emph{diffusive} distance, i.e., proportional to the square root of time. This represents the major hurdle we need to solve in the present paper. In fact, as shown in \cite{KS05}, information propagates \emph{ballistically}, i.e., linearly in time. The latter means that we cannot obtain the desired results by a simple argument based on the non-superdiffusivity property of the lattice gas particles derived in \cite{GdHNOS09}. It turns out that this near independence can be achieved by a control on the spreading velocity of a rumour in a sparse and moving population. To do so we exploit the work in \cite{GN10}, which derives an upper bound on the velocity of front propagation, from which we deduce that \emph{essentially} no constructive information can travel between distant volumes on appropriate space-time scales. In particular, we will see that moderately large volumes are \emph{not} independent on time scales of the same order as the nucleation time. Nevertheless, we can overcome this obstacle because the independence is needed on a shorter time scale only. Indeed, since the transitions of the reduced Markov chain introduced in \cite[Section 1.3.2]{BGdHNOS23} take place on this shorter time scale, we can project our dynamics on a product of finite state spaces with independent parallel evolutions. This construction, which is developed in Section \ref{sec:upplarge}, represents the main novelty of the present paper. It, in turn, implies that nucleation in large volumes is \emph{homogeneous}, i.e., occurs independently and with equal probability in disjoint moderate volumes that are essentially independent.
	
We refer the reader to \cite{BGdHNOS23} for an extensive description of what is achieved in the three papers together (see \cite[Items (1)--(7) in Section 1.1]{BGdHNOS23}) and what are the key challenges associated with Kawasaki dynamics in large volumes (see \cite[Remarks 1.1--1.3]{BGdHNOS23}). For more background on metastability for interacting particle systems, we refer the reader to the monographs \cite{OV04}, \cite{BdH15} and references therein.  
	
The innovative contributions of the present paper can thus be outlined as follows.
\begin{enumerate}
\item
We provide a detailed description of the escape from metastability for Kawasaki dynamics in large volumes. As illustrated by Theorem \ref{thm:tube} and proved in Section \ref{sec:nocoal}, the escape from metastability occurs via nucleation and \emph{not} via coalescence. The nucleation is \emph{homogeneous}, in the sense that a supercritical droplet appears for the first time more or less independently and with equal probability in disjoint boxes of moderate volume, and trigger the nucleation. More precisely, the configurations in moderately large boxes behave as if they are essentially independent (see Section \ref{sec:upplarge}) and as if the surrounding gas is \emph{ideal}. 
\item
We show that, even on time scales where far-away regions in the gas can interact with each other, information \emph{cannot} travel between them in such a way as to destroy the quasi-independence of the evolutions in these regions. In fact, the growing or shrinking of droplets occur as rare events in the middle of frequent recurrences to local configurations on much shorter time scales, namely, those describing the interaction between the droplets and the gas. See Section \ref{sec:upplarge}.
\item
Information travels at a positive speed as shown in \cite{KS05}. The upper bound on the speed derived in \cite{GN10} depends on the \emph{density} of the gas and allows us to show that, in the low-density limit, far-away volumes that are moderately large are essentially independent. The latter allows us to use the results in the second paper \cite{BGdHNOS23}, valid for moderately large boxes, to tackle the nucleation on large boxes and complete the escape from metastability at low temperature and low density for conservative dynamics.
\end{enumerate}
	
The conservative nature of Kawasaki dynamics prevents us from moving to higher temperatures (and possibly all the way up to the critical temperature), because of \emph{lack of monotonicity} (i.e., \emph{attractiveness}), as for instance employed in \cite{SS98} and \cite{BGM20,GMV}, for non-conservative Glauber dynamics. These papers show that, whatever mathematical techniques are developed, the crucial tool is to first understand how subcritical and supercritical droplets evolve in the low-temperature set-up. This serves as the starting point for tackling the problem of nucleation for conservative dynamics at all subcritical temperatures and is what we achieve in the present paper.

	
\subsection{Kawasaki dynamics}


\paragraph{$\bullet$ Hamiltonian, generator and equilibrium.}	

Let $\b>0$ denote the inverse temperature. Let $\L_\b \subset \mathbb Z^2$ be the square box with volume
\begin{equation}
\label{Lvol}
|\L_\b| = \ee^{\T\b}, \qquad \T>0,
\end{equation}
centered at the origin with periodic boundary conditions. With each $x\in \L_\b$ associate an occupation variable $\h(x)$, assuming the values $0$ or $1$. A lattice gas configuration is denoted by  $\h \in \cX_\beta= \{0,1\}^{\L_\b}$. With each configuration $\eta$ associate an energy given by the Hamiltonian
\begin{equation}
\label{Ham*}
H(\h) = -U \sum_{\{x,y\}\in \L_\b^*}\h(x)\h(y),
\end{equation}
where $\L_\b^*$ denotes the set of bonds between nearest-neighbour sites in $\L_\b$, i.e., there is a \emph{binding energy} $-U<0$ between neighbouring particles. Let 
\begin{equation}
\label{numpart*}
|\eta|=\sum_{x\in\Lambda_\beta}\h(x)
\end{equation}
be the number of particles in $\Lambda_\beta$ in the configuration $\h$,  and let
\begin{equation}
\label{Npart}
\cX_N = \{\h\in\cX_\beta\colon\,|\eta|=N\}
\end{equation}
be the set of configurations with $N$ particles. 

We define Kawasaki dynamics as the continuous-time Markov chain $X=(X(t))_{t\geq0}$ with state space $\cX_N$  given by the generator
\begin{equation}
\label{gendef}
(\cL f)(\h)=\sum_{\{x,y\}\in\L_\b^*}c(x,y,\h)[f(\h^{x,y})-f(\h)],
\qquad \eta\in{\cal X}_{\beta},
\end{equation}
where
\begin{equation}
\label{confexch}
\h^{x,y}(z)= \left\lbrace\begin{array}{lcl}
\h(z) &{\rm if} & z \neq x,y, \\
\h(x) &{\rm if} & z=y, \\
\h(y) &{\rm if} & z=x,  \end{array}\right.
\end{equation}
and
\begin{equation}
\label{rate}
c(x,y,\h)= \ee^{-\b [H(\h^{x,y})-H(\h)]_+}.
\end{equation}
Equations~\eqref{gendef}--\eqref{rate} represent the standard \emph{Metropolis dynamics} associated with $H$, and is \emph{conservative} because it preserves the number of particles, i.e., $|X(t)|=|X(0)|$ for all $t>0$. The \emph{canonical Gibbs measure} $\n_N$ defined as
\begin{equation}
\label{nuN}
\n_N(\h) = {\ee^{-\b H(\h)}\,1_{\cX_N}(\h) \over Z_N},
\qquad Z_N = \sum_{\h\in \cX_N} \ee^{-\b H(\h)},
\qquad \h \in \cX_\beta,
\end{equation}
is the reversible equilibrium of this stochastic dynamics for any $N$:
\begin{equation}
\label{rev}
\n_N(\h) c(x,y,\h) = \n_N(\h^{x,y}) c(x,y,\h^{x,y}).
\end{equation}
With $\Delta>0$ the \emph{activation energy} per particle, the grand-canonical Gibbs measure $\mu$ is defined by
\[
\mu(\eta) = \dfrac{\ee^{-\beta [H(\eta) + \Delta |\eta|]}}{Z}, \qquad Z 
= \sum_{\h\in \cX_\beta} \ee^{-\beta [H(\eta) + \Delta |\eta|]}, \qquad \eta\in\cX_\beta.
\]
This models the presence of an external reservoir that keeps the density of particles in $\Lambda_\beta$ fixed at $\ee^{-\beta\Delta}$.


\paragraph{$\bullet$ Subcritical, critical and supercritical droplets.}

The initial configuration is chosen according to the \emph{grand-canonical Gibbs measure} restricted to the set of subcritical droplets. More precisely, denote by 
\begin{equation}
\label{def:lc}
\ell_ c = \Big\lceil \frac{U}{2U-\Delta}\Big\rceil
\end{equation}
the critical length introduced in \cite{dHOS00a} for the \emph{local model} where $\Lambda_\beta=\Lambda$ does not depend on $\beta$ (see Fig.~\ref{fig-cancrit}).  The energy of the critical droplet in the local model equals
\begin{equation}
\Gamma = - U[(\ell_c-1)^2+\ell_c(\ell_c-2)+1] +\Delta [\ell_c(\ell_c-1)+2].
\end{equation}

\begin{figure}
\setlength{\unitlength}{0.20cm}
\begin{picture}(15,15)(-12,0)
\qbezier[40](10,-4)(20,-4)(30,-4)
\qbezier[40](10,-4)(10,3)(10,15)
\qbezier[40](10,15)(20,15)(30,15)
\qbezier[40](30,-4)(30,10)(30,15)
\put(15,0){\line(1,0){10}}
\put(15,11){\line(1,0){10}}
\put(15,0){\line(0,1){11}} 
\put(26,12){\line(1,0){1}}
\put(27,12){\line(0,1){1}}
\put(27,13){\line(-1,0){1}}
\put(26,13){\line(0,-1){1}}
 \put(25,0){\line(0,1){5}}
\put(25,6){\line(0,1){5}}
\put(26,5){\line(0,1){1}}
\put(25,5){\line(1,0){1}}
\put(25,6){\line(1,0){1}}
\qbezier[5](25,5)(25,5.5)(25,6)
\put(13,5.5){$\ell_c$}
\put(18.5,-2){$\ell_c-1$}
\put(28.5,-3){$\Lambda$}
\end{picture}
\vspace{1cm}
\caption{\small A critical droplet in a finite volume $\Lambda$: a protocritical droplet, consisting of an $(\ell_c-1) \times \ell_c$ quasi-square with a single protuberance attached to one of the longest sides, and a free particle nearby. When the free particle attaches itself to the protuberance, the droplet becomes supercritical. \normalsize}
\label{fig-cancrit}
\end{figure}
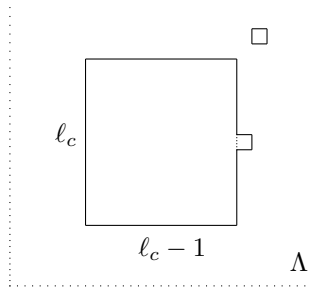 
By defining a cluster as a connected component of nearest-neighbours particles, define
\begin{equation}
\label{def:R}
{\cal R} = \big\{\eta\in{\cal X}_\beta\colon\,\hbox{all clusters of } \eta \hbox{ have volume at most } \ell_c(\ell_c-1)+2\big\}
\end{equation}
and put
\begin{equation}
\label{muR}
\mu_{\cal R}(\h) = \frac{\ee^{-\b[H(\h)+\Delta |\eta|]}}{Z_{\cal R}}\,1_{\cal R}(\eta), \qquad \h\in\cX_\beta,
\end{equation}
where
\begin{equation}
\label{ZR}
Z_{\cal R} = \sum_{\h\in {\cal R}} \ee^{-\b [H(\h)+\Delta |\eta|]}
\end{equation}
is the normalising partition sum. Write $P_\eta$ for the law of $X$ given $X(0)=\eta$. The initial configuration $X(0)$ is drawn from $\mu_{\cal R}$, i.e.,  the initial law is $P_{\mu_{{\cal R}}} = \int_{\cX_\beta} \mu_{{\cal R}}(\dd \eta) P_\eta$.

We will be interested in the regime $\D \in (U,2U)$ and $\b \to\infty$, which corresponds to \emph{metastable behaviour}. In this regime, droplets with side length smaller than $\ell_c$ have a tendency to shrink, while droplets with a side length larger that $\ell_c$ have a tendency to grow \cite{dHOS00a}. We will refer to the former as \emph{subcritical droplets} and to the latter as \emph{supercritical droplets}. To avoid trivialities and ensure that $\ell_c>2$, we assume that
\begin{equation}
\label{metreg}
\D \in \left(\tfrac32 U,2U\right).
\end{equation}



\subsection{Main theorems}


\paragraph{$\bullet$ Nucleation time.}

Let 
\begin{equation}
\tau_{\cR^c} = \inf\{t \geq 0\colon\, X(t) \notin \cR\}.
\end{equation}

\begin{theorem}{\bf [Homogeneous nucleation time]}
\label{thm:escape}
Subject to \eqref{metreg}, for every $\D<\T <\Gamma-(2\Delta-U)$ and every $\d>0$, 
\begin{equation}
\lim_{\beta\to\infty}\frac{1}{\beta}\ln
P_{\mu_{\cR}}\Biggl(\tau_{\cR^c}\notin\Biggl[
\frac{e^{\Gamma\beta}}{|\Lambda_\beta|}e^{-\delta\beta},
\frac{e^{\Gamma\beta}}{|\Lambda_\beta|}e^{\delta\beta}\Biggr]\Biggr)<0.
\end{equation}
\end{theorem}


\paragraph{$\bullet$ Tube of typical trajectories.}

We define {\it quasi-squares} as clusters of sizes $\ell_1\times\ell_2$ in the set
\[
\hbox{QS}=\{(\ell_1,\ell_2)\in\N^2: \ell_1\leq \ell_2\leq \ell_1+1\}
\]
and let $\lambda(\b)$ be an unbounded but slowly increasing function of $\beta$ satisfying
\begin{equation}
\label{def:lambda}
\lambda(\b)\log\lambda(\b)=o(\log\b), \qquad \beta \to \infty,
\end{equation}
e.g.\ $\lambda(\b)=\sqrt{\log\b}$. Subject to \eqref{metreg}, for $\D<\T<\Gamma-(2\Delta-U)$ and starting from the initial measure $\mu_{\cR}$, we will prove that with a probability that is exponentially close to $1$ the system nucleates a growing cluster that visits the full increasing sequence of quasi-squares, and after the exit from $\cR$ grows in the same way up to a $\sqrt{\lambda(\beta)/8}\times\sqrt{\lambda(\beta)/8}$-square. This square is reached after the exit from $\cR$ in a time that is exponentially smaller than the time needed to reach $\cR^c$. To formally state the latter, let $\theta=2\Delta-U-\gamma$, where
\begin{equation}\label{def:gamma}
\gamma = (\Delta - U) - (\ell_c - 2)(2U-\Delta) > 0.
\end{equation}
We denote by $\Lambda_{\cR^c}$ the box with volume $\ee^{\theta\beta}$ centered at the baricenter of the first cluster of volume $\ell_c(\ell_c+1)+3$. Note that the box $\Lambda_{\cR^c}$ is defined at time $\tau_{\cR^c}$. For $(\ell_1,\ell_2)\in\hbox{QS}$ with $\ell_1<\ell_c$, define
\[
\tau_{(\ell_1,\ell_2)}^{last}=\sup\{t\leq\tau_{\cR^c}\colon\, X(t) \hbox{ contains a } \ell_1\times\ell_2 \hbox{ quasi-square cluster in } \Lambda_{\cR^c}\},
\]
with the convention that $\sup\emptyset=-\infty$. Similarly, for $\ell_1\geq\ell_c$, define
\[
\tau_{(\ell_1,\ell_2)}=\inf\{t\geq\tau_{\cR^c}: X(t) \hbox{ contains a } \ell_1\times\ell_2 \hbox{ quasi-square cluster in } \Lambda_{\cR^c}\},
\]
with the convention that $\inf\emptyset=\infty$. Finally, for $\delta>0$, define the event
\[
\begin{array}{ll}
T_\delta 
&= \Bigl\{ \tau_{\cR^c}-\ee^{(\theta+\delta)\beta}< \tau_{(2,2)}^{last} 
< \tau_{(2,3)}^{last}< \ldots <\tau_{(\ell_c-1,\ell_c)}^{last} < \tau_{\cR^c}<\tau_{(\ell_c,\ell_c)} \\
&\qquad \qquad \qquad < \ldots < \tau_{\left(\sqrt{\lambda(\beta)/8},\sqrt{\lambda(\beta)/8}\right)}  
< \tau_{\cR^c}+\ee^{(2\Delta-U+\delta)\beta} < +\infty \Bigr\}.
\end{array}
\]

\begin{theorem}{\bf [Tube of typical trajectories]}
\label{thm:tube}
Subject to \eqref{metreg}, for any $\delta>0$,
\begin{equation}
	\lim_{\beta\to\infty}\frac{1}{\beta}\ln
	P_{\mu_{\cR}}\left( T_\delta^c \right)<0.
\end{equation}
\end{theorem}
In the proof of Theorems \ref{thm:escape} and \ref{thm:tube} the parameter $\Theta$ will play a crucial role. Indeed, in Section \ref{sec:upplarge} we will see that the analysis for {\it large volumes}, i.e., $\Theta>\theta$, can be reduced to that for {\it moderate volumes}, i.e., $\Theta\leq\theta$, by using their near-independence, where the parameter $\theta$ can be interpreted as follows. Let $r(\ell_1,\ell_2)$ be the \emph{resistance} of the $\ell_1\times \ell_2$ quasi-square with $1\leq \ell_1\leq \ell_2$ given by (see Fig.~\ref{fig:resistenza})
\begin{equation}
	\label{def:res}
	r(\ell_1,\ell_2) = \min\{(2U - \Delta) \ell_1 - U + 2\Delta - U, 2\Delta - U\}.
\end{equation}

\begin{figure}[htbp]
	\centering
	\includegraphics[width=0.8\textwidth]{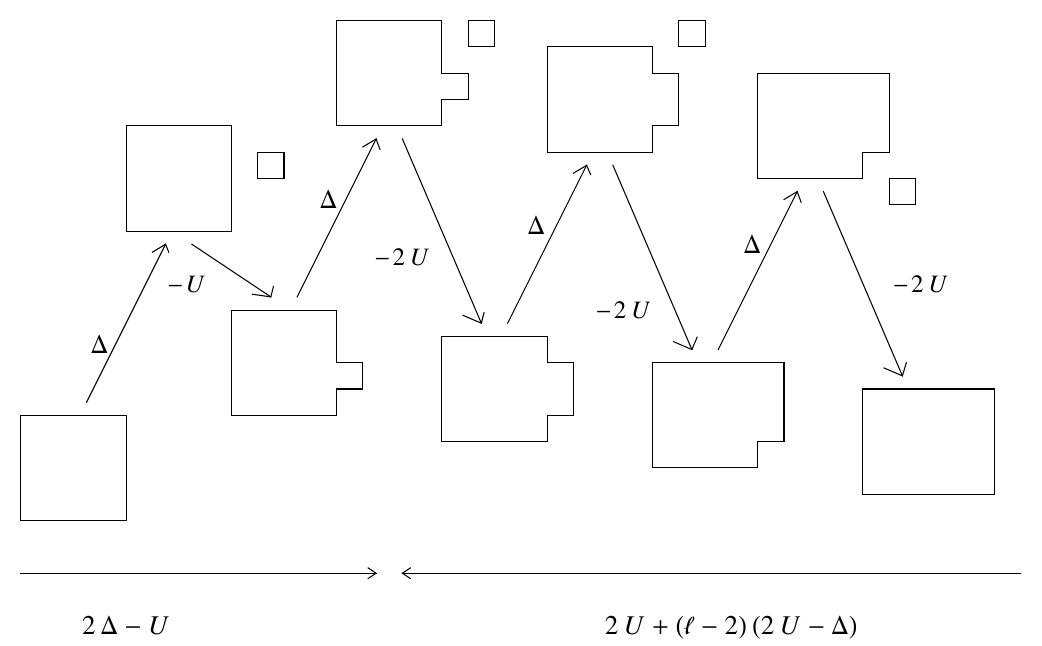}
	\caption{\small Cost of adding or removing a row of length $\ell$ in a finite volume. \normalsize}
	\label{fig:resistenza}
\end{figure}

\noindent
Note that $\theta$ can be viewed as the resistance of the largest subcritical quasi-square. Since this quasi-square has sizes $(\ell_c - 1) \times \ell_c$, we have $\theta = 2U + (\ell_c-3)(2U-\Delta)$, which can be written as $\theta=2\Delta-U-\gamma$, where $\gamma$ is defined in \eqref{def:gamma}.

\begin{remark}
{\rm With the techniques developed in \cite{BGdHNOS23} it is possible to show that, as in \cite{GOS04}, the growing cluster in Theorem~\ref{thm:tube} triggering the nucleation is {\it wandering}, i.e., when the cluster has size $\ell_1\times\ell_2$ with $|\ell_1-\ell_2| \leq 1$ it exits any finite box around it with a volume that does not depend on $\beta$ within a time of order $\ee^{r(\ell_1,\ell_2)\beta}$.}
\end{remark}


\subsection{Conclusion}

{\bf 1.}
Theorems~\ref{thm:escape}--\ref{thm:tube} provide a detailed description of the \emph{escape from metastability} for Kawasaki dynamics in moderate and large volumes, and concludes the work initiated in \cite{GdHNOS09} and \cite{BGdHNOS23}. The asymptotics of the nucleation time is identified on a time scale that is exponential in $\beta$ and depends on an {\it entropic factor} related to the size of the box: $\tau_{\cR^c}$ is typically of order $\ee^{\Gamma\beta}/|\Lambda_\beta|$. We find that the probability to deviate from the lower bound is exponentially small in $\beta$, while the probability to deviate from the upper bound is $\SES(\beta)$, where we write $\SES(\beta)$ for any function of $\beta$ that decays to zero faster than any exponential of $\beta$. This asymmetry is due to the fact that nucleation after a given time can only happen when an exponentially large number of particles does not manage to trigger the nucleation before that time.

\medskip\noindent
{\bf 2.}
Our theorems only concern the \emph{initial phase} of the nucleation, until the critical droplet grows into a droplet that is roughly $\sqrt{\lambda(\beta)}$ times the size of the critical droplet. They provide no information on what happens \emph{afterwards}, when the droplet grows even further and becomes macroscopically large. In that regime the gas around the droplet becomes depleted, smaller droplets move around and coalesce into larger droplets, etc. It remains a major challenge to describe what precisely happens in this regime, which lies \emph{beyond} metastability.    

\paragraph{Outline.}
The rest of the paper is organised as follows. In Section \ref{sec:notation} we recall some definitions and notations that will be used throughout the paper. In Section \ref{sec:lowbound} we provide the lower bound for the nucleation time for all $\D<\Theta<\Gamma-(2\D-U)$, while in Section \ref{sec:uppsmall} (respectively, Section \ref{sec:upplarge}) we provide the upper bound for the nucleation time and the tube of typical trajectories for all $\Delta<\Theta\leq\theta$ (respectively, $\theta<\Theta<\Gamma-(2\D-U)$). In Sections \ref{sec:lowbound} and \ref{sec:uppsmall} we use technical arguments that are explained in more detail in \cite{BGdHNOS23}, and focus on the relevant estimates only. The lower bound for the nucleation time can be obtained for arbitrarily large volumes. The upper bound is much harder.


\section{Definitions and notations}
\label{sec:notation}

In this section we introduce some definitions and notations that will be needed throughout the sequel. 

\begin{definition}
$\mbox{}$
{\rm 
\begin{itemize}
\item[1.] 
As in \cite{GdHNOS09}, \cite{BGdHNOS23}, $\alpha$ and $d$ are two positive parameters that can be chosen as small as desired, and $\lambda(\b)$ is an unbounded but slowly increasing function of $\beta$ that satisfies \eqref{def:lambda}. Moreover, $C^{\star}$ is a positive parameter that can be chosen as large as desired. Once chosen, $\alpha$, $d$, $\lambda$ and $C^{\star}$ are fixed. We write $O(\delta)$, $O(\alpha)$ and $O(d)$ for quantities with an absolute value that can be bounded by a constant times $|\delta|$, $|\alpha|$ and $|d|$, for small enough values of these parameters. We write $O(\delta,\alpha,d)$ for the sum of three such quantities.
\item[2.] 
We use short-hand notation for a few quantities that depend on the previous parameters $\Delta \in (\tfrac32 U,2U)$, $\Theta \in (\Delta,\Gamma-(2\Delta-U))$ and the new parameters $\alpha$, $d$. Recall that
\begin{equation}
\label{parameters}
\theta=2\Delta-U-\gamma, \quad \gamma=\Delta-U-(\ell_c-2)\epsilon \quad \hbox{and} \quad \ell_c=\Big\lceil\frac{U}{\epsilon}\Big\rceil, \hbox{ with } \epsilon=2U-\Delta.
\end{equation}
We set
$$
D=U+d, \qquad \Delta^+=\Delta+\alpha,
$$
and abbreviate
\begin{equation}
\label{Spardef}
S=\dfrac{4\Delta-\theta}{3}-\alpha.
\end{equation}
For $C>0$, write $T_C$ for the time scale $T_C=\ee^{C\beta}$.
\item[3.]
For convenience we identify a configuration $\eta\in\cX_\beta$ with its support $\hbox{supp}(\eta)=\{z\in\Lambda_\beta: \eta(z)=1\}$ and write $z\in\eta$ to indicate that $\eta$ has a particle at $z$. For $\h\in\cX_\beta$, denote by $\h^{cl}$ the {\it clusterised part of} $\h$:
\begin{equation}
\label{def:cluster}
\h^{cl} = \{z\in\h\colon\,\parallel z-z'\parallel =1 \hbox{ for some } z'\in\h\}.
\end{equation}
Call {\it clusters of} $\h$ the connected components of the graph drawn on $\h^{cl}$ obtained by connecting nearest-neighbour sites that are not a singleton.
\item[4.]
Denote by $B(z,r)$, $z\in\R^2$, $r>0$, the open ball on $\R^2$  with center $z$ and radius $r$ in the $\ell_{\infty}$-norm defined by 
\begin{equation}
\label{def:norm}
||\cdot||_{\infty}\colon\, (x,y)\in\R^2 \mapsto |x|\vee|y|.
\end{equation}
We denote by $\dist(\cdot,\cdot)$ the distance induced by this norm. 
\item[5.] 
The closure of $A\subset\R^2$ is denoted by $\overline{A}$. For $A\subset\Z^2$ and $s>0$, put
\begin{equation}\label{eq:orlo}
[A,s] = \bigcup_{z\in A}\overline{B(z,\ee^{\frac{s}{2}\beta})}\cap\Z^2.
\end{equation}
Call $A$ a rectangle on $\Z^2$ if there are $a,b,c,d\in\R$ such that
\begin{equation}
A = [a,b] \times [c,d] \cap\Z^2.
\end{equation}
Write $\hbox{RC(A)}$ to denote the intersection of all the rectangles on $\Z^2$ containing $A$, called the {\it circumscribed rectangle} of $A$. Denote by ${\mathscr{R}}$ the set of all finite collections of rectangles on $\Z^2$.
\item[7.] 
Given $\sigma\geq0$ and $\bar S=\{R_1,\ldots,R_{|S|}\}\in{\mathscr{R}}$, two rectangles $R$ and $R'$ in $\bar S$ are said to be in the same equivalence class if there exists a finite sequence $R_1,\ldots,R_k$ of rectangles in $\bar S$ such that
$$
R=R_1, \quad R'=R_k, \quad  \dist(R_j,R_{j+1})<\sigma \quad \forall\,\, 1 \leq j < k.
$$
Let $C$ be the set of equivalent classes, define the map
$$
\bar g_\sigma\colon\, \bar S\in{\mathscr{R}} \mapsto
\Bigg\{\hbox{RC}\Bigg(\bigcup_{j \in c}R_j\Bigg)\Bigg\}_{c\in C}\in{\mathscr{R}},
$$
and let $(\bar g_\sigma^{(k)})_{k\in\N_0}\in {\mathscr{R}}^{\mathbb{N}}$ be the sequence of iterates of $\bar g_\sigma$. Define
\begin{equation}
\label{def:mapg}
g_\sigma(\bar S)= \lim_{k\rightarrow\infty} \bar g_\sigma^{(k)}(\bar S).
\end{equation}
As discussed in \cite{G09}, the sequence $(\bar g_\sigma^{(k)}(\bar S))_{k\in\N_0}$ ends up being a constant, so the limit is well defined.
\end{itemize}
}\hfill$\spadesuit$
\end{definition}

\begin{figure}
	\centering
	\begin{tikzpicture}[scale=0.25,transform shape]		
		\draw [gray, fill=gray, semithick, even odd rule] 
		(20,10) rectangle (37,20) (23,12) rectangle (34,18);
		\draw [gray, fill=gray] (29,14) rectangle (30,15);
		\node at (30.7,15.2){\huge{10}};
		\draw [gray, fill=gray] (19,9) rectangle (20,10);
		\node at (20.5,8.8){\huge{15}};
		\draw [gray, fill=gray] (18,8) rectangle (19,9);
		\node at (19.5,7.8){\huge{14}};
		\draw [gray, fill=gray] (17,7) rectangle (18,8);
		\node at (16.5,8.2){\huge{13}};
		\draw [gray, fill=gray] (16,6) rectangle (17,7);
		\node at (15.5,7.2){\huge{12}};
		\draw [gray, fill=gray] (15,5) rectangle (16,6);
		\node at (14.5,6.2){\huge{11}};
		\draw [gray, fill=gray, semithick] 
		(15,5) rectangle (11,-3);
		\draw [gray, fill=gray, semithick] 
		(11,5) rectangle (10,6);
		\node at (9.7,4.7){\huge{4}};
		\draw [gray, fill=gray, semithick] 
		(10,6) rectangle (9,7);
		\node at (10.3,6.4){\huge{5}};
		\draw [gray, fill=gray, semithick] 
		(11,7) rectangle (10,8);
		\node at (11.3,8.4){\huge{2}};
		\draw [gray, fill=gray, semithick] 
		(9,7) rectangle (8,8);
		\node at (7.7,8.4){\huge{1}};
		\draw [gray, fill=gray, semithick] 
		(9,5) rectangle (8,6);
		\node at (7.7,4.8){\huge{3}};
		\draw [gray, fill=gray, semithick] 
		(15,-3) rectangle (16,-4);
		\node at (15.8,-2.5){\huge{9}};
		\draw [gray, fill=gray, semithick] 
		(15,-5) rectangle (16,-6);
		\node at (14.5,-5){\huge{8}};
		\draw [gray, fill=gray, semithick] 
		(15,-6) rectangle (10,-9);
		\draw [gray, fill=gray, semithick] 
		(18,-6) rectangle (21,-9);
		\draw [gray, fill=gray, semithick] 
		(18,-6) rectangle (17,-5);
		\node at (16.8,-6.5){\huge{7}};
		\draw [gray, fill=gray, semithick] 
		(18,-4) rectangle (17,-3);
		\node at (18.5,-3.6){\huge{6}};
		\draw [gray, fill=gray, semithick] 
		(18,-3) rectangle (23,0);
		\draw [gray, fill=gray, semithick] 
		(32,-1) rectangle (33,0);
		\node at (33.7,0.5){\huge{16}};
	\end{tikzpicture}
	
	\vskip 0 cm
	\caption{\small Each particle is represented by a unit square. Particles 1--5 and 16 are free, particles 6--9, 10, 11--15 are not free. All other particles are clusterised. \normalsize}
	\label{fig:freeparticle}
\end{figure}
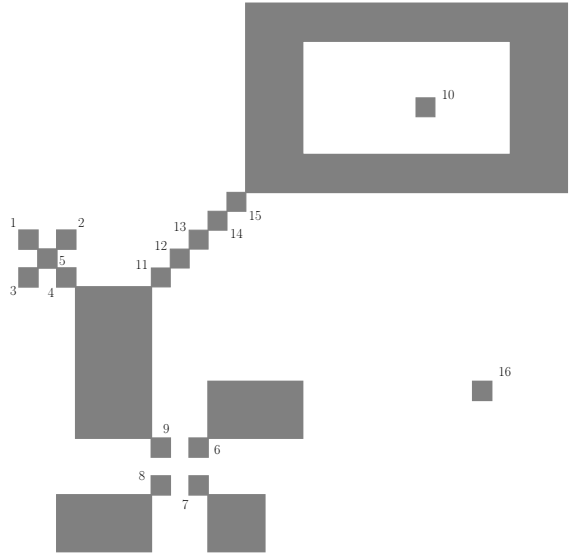

As in \cite{GdHNOS09}, \cite{BGdHNOS23}, the notion of \emph{active} and \emph{sleeping} particles will be crucial. For the precise definition we refer to \cite[Section 4.3.1]{BGdHNOS23} and \cite[Section 2.2]{GdHNOS09}. The division of particles into active and sleeping is related to the notion of free particles. Intuitively, a particle is {\it free} if it does not belong to a cluster and can be moved to infinity without clusterisation, i.e., by moving non-clusterised particles only (see Figure \ref{fig:freeparticle}). For $t>\ee^{D\beta}$, a particle is said to be {\it sleeping} at time $t$ if it was not free during the time interval $[t-\ee^{D\beta},t]$. Non-sleeping particles are called {\it active}. Note that being active or sleeping depends on the history of the particle.

In \cite[Definition 2.3]{BGdHNOS23} we introduced a subset of configurations $\cX^*\subset\cX_\beta$ satisfying certain regularity properties. We referred to this subset as the {\it typical environment}.

\begin{definition}
\label{def:recurrence}
{\rm For any time $t\geq0$, given a configuration $\eta_t=X(t)\in{\cal X}_{\beta}$ and the collection $\bar\Lambda(t) = (\bar\Lambda_i(t))_{1 \leq i \leq k(t)}$ of finite boxes in $\Lambda_\beta$ constructed in \cite[Definition 1.1]{BGdHNOS23}, we say that $\eta_t$ is 0-reducible (respectively, $U$-reducible) if for some $i$ the local energy $\bar{H}_i(\bar\eta_i)$ given in \cite[Section 2.3]{BGdHNOS23} can be reduced along the dynamics with constant $\bar\Lambda(t)$ without exceeding the energy level $\bar H_{i}(\bar\eta_i)+0$ (respectively, $\bar H_{i}(\bar\eta_i)+U$). If $\eta_t$ is not $0$-reducible or $U$-reducible, then we say that $\eta_t$ is $0$-irreducible or $U$-irreducible, respectively. We define
$$
\begin{array}{ll}
{\cal X}_0
&=\{\eta_t\in{\cal X}^*\colon\,\eta_t \hbox{ is $0$-irreducible}\}, \\[0.2cm]
{\cal X}_U
&=\{\eta_t\in{\cal X}_0\colon\,\eta_t \hbox{ is $U$-irreducible}\}, \\[0.2cm]
{\cal X}_D 
&=\{\eta_t\in{\cal X}_U\colon\,\hbox{all the particles in } \Lambda(t) \hbox{ are sleeping}\}, \\[0.2cm]
{\cal X}_S
&=\{\eta_t\in{\cal X}_D\colon\,\hbox{each box of volume } \ee^{S\beta} \hbox{ contains three active particles at most}\}, \\[0.2cm]
{\cal X}_{\Delta^+}
&=\left\{\eta_t\in{\cal X}_S\colon\,
\begin{array}{ll}
\bar\eta_t \hbox{ is a union of at most } \lambda(\beta) \hbox{ quasi-squares with} \\
\hbox{no particle inside } \bigcup_i[\bar\Lambda_i(t),\Delta-\alpha] 
\hbox{ except for those} \\
\hbox{in the quasi-squares, one for each local box } \bar\Lambda_i(t) 
\end{array}
\right\}, \vspace{0.2cm} \\
{\cal X}_E
&=\{\eta\in{\cal X}_{\Delta^+}\colon\, \eta \hbox{ has no quasi-square}\},
\end{array}
$$
where $[\bar\Lambda_i(t),\Delta-\alpha]$ are the boxes of volume $\ee^{(\Delta-\alpha)\beta}$ with the same center as $\bar\Lambda_i(t)$.}\hfill$\spadesuit$
\end{definition}
Define the set
\begin{equation}
\label{def:R'}
{\cal R}':=\left\{\eta\in{\cal X}_\beta\colon\,
\begin{array}{ll}
\hbox{all clusters of } \eta \hbox{ have volume at most } \ell_c(\ell_c-1)+2 \\
\hbox{except for at most one cluster with volume less than } \tfrac18\lambda(\beta)
\end{array}
\right \},
\end{equation}
where $\lambda(\beta)$ satisfies \eqref{def:lambda}. For $C^{\star}>0$ large enough, our theorem about the tube of typical trajectories will hold up to time $T^{\star}$ defined as
\begin{equation}
\label{def:Tstella}
T^{\star} = \ee^{C^{\star}\beta}\wedge\min\{t \geq 0\colon\, X(t)\notin{\cal R'}\}.
\end{equation}
In \cite[Proposition 2.6]{BGdHNOS23} we showed that if our system is started from the restricted ensemble $\mu_{{\cal R}}$, then with probability $\SES$ it escapes from $\cX^*$ within time $T^{\star}$. Thus, effectively our system is confined to $\cX^*$, which is needed for certain arguments later on.


\section{Lower bound}
\label{sec:lowbound}

In this section we prove the lower bound for the time $\tau_{\cR^c}$ stated in Theorem \ref{thm:escape} for any $\Delta<\Theta<\Gamma-(2\Delta-U)$. 

Let $\Pi^k=(\Pi_i^k)_{i<n}$, with $k<4$ and $n=|\Lambda_\beta|/(2\ell_c)^2$, be the four partitions of $\Lambda_\beta$ into boxes $\Lambda_i$ with side length $2\ell_c$ centered on a square grid of side length $\ell_c$. Given such a partition $\Pi^k$ and a configuration $\eta\in\cX_\beta$, for $i<n$ define $\eta^{k,\mathrm{in}}_i$ as the configuration given by all the clusters in $\eta$ having a non-empty intersection with $\Lambda_i$ and all the free particles in $\Lambda_i$. Note that the support of $\eta^{k,\mathrm{in}}_i$ may not be completely contained in $\Lambda_i$. Indeed, clusters of $\eta$ may intersect $\Lambda_i$, but not be contained inside. Moreover, define $\eta_i^{k,\mathrm{out}}=\eta - \eta^{k,\mathrm{in}}_i$. Note that, since $\eta_i^{k,\mathrm{out}}$ and $\eta_i^{k,\mathrm{in}}$ have disjoint support, $H(\eta)=H(\eta_i^{k,\mathrm{in}})+H(\eta_i^{k,\mathrm{out}})$. Define
\[
\begin{aligned}
A_i^k &= \Big\{ \eta\in\cR\colon\, |\eta_i^{k,\mathrm{in}}|\geq\ell_c(\ell_c-1)+2 \hbox{ and }\\ 
&\qquad \eta_i^{k,\mathrm{in}} \hbox{ has no cluster with more than } \ell_c(\ell_c-1)+1 \hbox{ particles} \Big\}.
\end{aligned}
\]
If $X$ exits from $\cR$ within time $t$, then there exists $0\leq s\leq t$ such that $X(s)\in\bigcup_{k<4}\bigcup_{i<n} A_i^k$. Considering the Poisson process along which the process $X$ is updated, and letting $T_-=\ee^{(\Gamma-\Theta-\delta)\beta}$ with $\delta>0$, we get
\[
P_{\mu_\cR} (\tau_{\cR^c}<T_-) \leq \SES + \sum_{k<4}\sum_{t<2T_-}\sum_{i<n} \mu_{\cR}(A_i^k).
\]
Note that
\[
\mu_{\cR}(A_i^k) = \dfrac{\displaystyle\sum_{\eta^{\mathrm{in}}\in A_i^{k,\mathrm{in}}} \ee^{-\beta[H(\eta^{\mathrm{in}})+\Delta|\eta^{\mathrm{in}}|]}\displaystyle\sum_{\eta^{\mathrm{out}}\sim \eta^{\mathrm{in}}} \ee^{-\beta[H(\eta^{\mathrm{in}})+\Delta|\eta^{\mathrm{in}}|]}}{\displaystyle\sum_{\eta^{\mathrm{in}}\in \cR_i^{k,\mathrm{in}}} \ee^{-\beta[H(\eta^{\mathrm{in}})+\Delta|\eta^{\mathrm{in}}|]}\displaystyle\sum_{\eta^{\mathrm{out}}\sim \eta^{\mathrm{in}}} \ee^{-\beta[H(\eta^{\mathrm{in}})+\Delta|\eta^{\mathrm{in}}|]}}.
\]
By estimating from below the denominator with the term associated to the configuration $\eta^{\mathrm{in}}=\eta^0$, where $\eta^0$ is the configuration that is equal to zero everywhere, and noting that $H(\eta^{\mathrm{in}})+\Delta|\eta^{\mathrm{in}}|\geq\Gamma$ for any $\eta^{\mathrm{in}}\in A_i^{k,\mathrm{in}}$, we get
\[
\mu_{\cR}(A_i^k) \leq |A_i^{k,\mathrm{in}}| \ee^{-\beta\Gamma}\leq \ee^{-\beta\big(\Gamma-\tfrac\delta2\big)},
\]
which implies the desired lower bound.


\section{Upper bound and tube: moderate volume}
\label{sec:uppsmall}

In this section we prove the upper bound for the time $\tau_{\cR^c}$ in Theorem \ref{thm:escape} and identify the nucleation pattern stated in Theorem \ref{thm:tube}, both in the case of moderate volume, i.e., $\Delta<\Theta\leq\theta$.

For  Theorem~\ref{thm:escape} it is enough to provide an upper bound on the exit time from $\cR$, i.e., we will prove that
\begin{equation}
	\label{Ptarget}
	P_{\mu_{\cal R}}(\tau_{{\cal R}^c} > \ee^{(\Gamma-\Theta+\delta)\beta})=\SES.
\end{equation}
Recall \cite[Proposition 2.6]{BGdHNOS23}, which says that
\begin{equation}\label{eq:step1}
P_{\mu_\cR}(\tau_{\cX_\beta\setminus\cX^*}\leq T^\star) = \SES,
\end{equation}
so that we can consider our dynamics confined to $\cX^*$. We will establish in Section \ref{sec:recvacuum} the recurrence property to the set $\cX_E\cup\cR^c$,
\begin{equation}
\label{eq:step2}
\sup_{\eta\in\cX^*}P_\eta(\tau_{\cX_E\cup\cR^c}\wedge\tau_{\cX_\beta\setminus\cX^*}>T_\theta\ee^{\delta\beta})=\SES.
\end{equation}
In Section \ref{sec:nuclevent} we will construct a nucleation event starting from $\cX_E$ to show that
\begin{equation}
\label{eq:step3}
\inf_{\eta\in\cX^*}P_\eta\left(\tau_{\cR^c}\wedge\tau_{\cX_\beta\setminus\cX^*}\leq T_\theta\ee^{\tfrac\delta3 \beta}\right) \geq \ee^{(\Theta-\Gamma)\beta} \ee^{\theta\beta} \ee^{-\tfrac\delta3 \beta}.
\end{equation}
By dividing the time interval $[0,\ee^{(\Gamma-\Theta)\beta} \ee^{\delta\beta}]$ into $\ee^{(\Gamma-\Theta-\theta)\beta} \ee^{\frac{2}{3}\delta\beta}$ intervals of length $T_\theta \ee^{\frac{\delta}{3}\beta}$, and using \eqref{eq:step1}, \eqref{eq:step2} and \eqref{eq:step3}, we get
\[
P_{\mu_{\cR}}(\tau_{\cR^c}>\ee^{(\Gamma-\Theta)\beta}\ee^{\delta\beta}) \leq \left( 1 - \ee^{(\Theta-\Gamma)\beta} \ee^{\theta\beta} \ee^{-\tfrac\delta3 \beta} - \SES \right)^{\ee^{(\Gamma-\Theta-\theta)\beta}\ee^{\tfrac23 \delta\beta}} = \SES.
\]

The proof of Theorem \ref{thm:tube} will be given in Sections \ref{sec:nocoal}--\ref{sec:suptube}. We first rule out the possibility of an exit from $\cR$ by coalescence  in Section \ref{sec:nocoal}. We subsequently construct the subcritical and the supercritical tube of typical trajectories in Sections \ref{sec:subtube} and \ref{sec:suptube} respectively,
which give us the complete geometrical description of the tube of typical paths stated in Theorem \ref{thm:tube} for $\Delta<\Theta\leq\theta$.


\subsection{Recurrence to vacuum}
\label{sec:recvacuum}

Here we show recurrence to the set $\cX_E \cup \cR^c$ within time $T_\theta\ee^{\delta\beta}$, where $T_{\theta} = \ee^{\theta\beta}$,
namely, we prove \eqref{eq:step2}. To this end, divide the time interval $[0, T_\theta \ee^{\delta\beta}]$ into $\ee^{\frac{3}{4}\delta\beta}$
intervals $I_j$ of length $T_\theta e^{\frac{\delta}{4}\beta}$. We have
$$
\begin{array}{ll}
\displaystyle\sup_{\eta\in\cX^*}P_\eta(\tau_{\cX_E\cup{\cal R}^c}
\wedge\tau_{\cX_\beta\setminus\cX^*}>T_\theta \ee^{\delta\beta})
&\leq \displaystyle\prod_{j<\ee^{\frac{3}{4}\delta\beta}}
\sup_{\eta\in\cX^*} P_\eta(\tau_{\cX_E\cup{\cal R}^c},
\tau_{\cX_\beta\setminus\cX^*}\notin I_j) \\
&=\Big(1-\displaystyle\inf_{\eta\in\cX^*}
P_\eta(\tau_{\cX_E\cup{\cal R}^c}
\wedge\tau_{\cX_\beta\setminus\cX^*}\leq T_\theta \ee^{\frac{\delta}{4}\beta})\Big)
^{\ee^{\frac{3}{4}\delta\beta}},
\end{array}
$$
where we use the strong Markov property for the stopping time $\tau_{\cX_E\cup{\cal R}^c}$. It suffices to prove that
\begin{equation}
\label{eq:rec2}
\inf_{\eta\in\cX^*}
P_\eta(\tau_{\cX_E\cup{\cal R}^c}
\wedge\tau_{\cX_\beta\setminus\cX^*}\leq T_\theta \ee^{\frac{\delta}{4}\beta}) \geq \ee^{-\frac{\delta}{4}\beta}.
\end{equation}
By the typical return time theorem (\cite[Theorem 1.5]{BGdHNOS23}), we know that the dynamics reaches the set $\cX_{\Delta^+}$ with probability $1-\SES$ within time $\ee^{(\Delta+\alpha+\delta)\beta}$ without exiting the environment $\cX^*$. If the configuration $X(\tau_{\cX_{\Delta^+}})$
contains a supercritical quasi square, namely, a $\ell_1\times\ell_2$ quasi-square with $\ell_1\geq\ell_c$, then the dynamics is already in $\cR^c$
and the event in \eqref{eq:rec2} is realized. Otherwise, all the quasi-squares in $X(\tau_{\cX_{\Delta^+}})$ are subcritical. Denote by $\ell_1$, $\ell_2$ the sizes of the smallest quasi-square in $X(\tau_{\cX_{\Delta^+}})$. By the typical update time theorem (\cite[Theorem 1.6]{BGdHNOS23}), we know that with probability $1-\SES$ the projected dynamics mantains the same quasi-square dimensions up to a coalescence event through successive visits in $\cX_{\Delta^+}$ within time $\ee^{(r(\ell_1,\ell_2)+\delta)\beta}\leq T_\theta \ee^{\delta\beta}$.
Therefore, by the typical transition theorem (\cite[Theorem 1.8]{BGdHNOS23}), we know that with probability at least $1-\ee^{-\delta\beta}$ the smallest quasi-square either decays into a smaller quasi-square or a {\it coalescence occurs at time $t$}, i.e.,  there exist two sleeping particles in different local boxes at time $t^-$ that are in the same local box at time $t$. In the first case, the volume of the droplets decreases. In the second case, two local boxes become too close to each other at time $t^-$, so that there is a non negligible probability that a coalescence between clusters appears. Thus, the number of droplets decreases. By iterating this argument for a non-exponential number of times, since the dynamics is in $\cX^*$ and hence the number of clusters is at most $\lambda(\beta)$, from \cite[Theorem 1.5]{BGdHNOS23} we obtain that the dynamics reaches the set $\cX_E \cup \cR^c$ within time $T_\theta\ee^{\delta\beta}$ with probability at least $\ee^{-\frac{\delta}{4}\beta}$, which proves \eqref{eq:rec2}.


\subsection{Exit from the set of subcritical configurations}
\label{sec:nuclevent}

Here we show that
\[
\inf_{\eta\in\cX_E}P_\eta\left( \tau_{\cR^c}\wedge\tau_{\cX_\beta\setminus\cX^*} \leq T_\theta \ee^{\tfrac\delta3 \beta} \right) \geq \ee^{(\Theta-\Gamma)\beta} \ee^{\theta\beta} \ee^{-\tfrac\delta3 \beta}.
\]
We estimate from below the probability that, starting from $\cX_E$, the system nucleates a single $2\times2$ square droplet anywhere in $\Lambda_\beta$ within time $T_\theta \ee^{\tfrac\delta6 \beta}$, and then, along the return times in $\cX_{\Delta^+}$, follows a sequence of quasi-squares of dimension $2\times3$, $3\times3, \ldots,(\ell_c-1)\times\ell_c$, waiting at most time $\ee^{r(\ell_1,\ell_2)\beta}\ee^{\tfrac\delta6 \beta}$ to go from an $\ell_1\times\ell_2$ quasi-square to the next one. By using the typical update time theorem \cite[Theorem 1.6]{BGdHNOS23} and the typical transition theorem \cite[Theorem 1.8]{BGdHNOS23}, we show that the probability to nucleate within time $T_\theta \ee^{\tfrac\delta6\beta}$ a $2\times2$ quasi-square anywhere in $\Lambda_\beta$ can be estimated from below by
\[
p_0 \geq \ee^{-r(0,0)\beta} \ee^{\theta\beta} \ee^{-\tfrac\delta6\beta},
\]
where $r(0,0)=4\Delta-2U-\Theta$. By using the typical update time theorem \cite[Theorem 1.6]{BGdHNOS23} and the typical transition theorem \cite[Theorem 1.9]{BGdHNOS23}, we show that the probability that a $\ell_1\times\ell_2$ quasi-square grows within time $\ee^{r(\ell_1,\ell_2)\beta} \ee^{\tfrac\delta6\beta}$ can be estimated from below by
\[
p_{(\ell_1,\ell_2)} \geq \ee^{-(2\Delta-U-r(\ell_1,\ell_2))\beta} \ee^{-\frac{\delta}{6\cdot 2^{\ell_1+\ell_2}}\beta}.
\]
Thus, we get
\[
\begin{array}{ll}
\displaystyle\inf_{\eta\in\cX_E}P_\eta\left( \tau_{\cR^c}\wedge\tau_{\cX_\beta\setminus\cX^*} \leq T_\theta \ee^{\tfrac\delta3 \beta} \right) &\geq  p_0 \displaystyle\prod_{\ell_1=2}^{\ell_c-1} \prod_{\ell_2=\ell_1}^{\ell_1+1} p_{(\ell_1,\ell_2)} \\
&\geq \ee^{-\left[ r(0,0) + \sum_{\ell_1=2}^{\ell_c-1}\sum_{\ell_2=\ell_1}^{\ell_1+1} (2\Delta-U-r(\ell_1,\ell_2)) \right]\beta} 
\ee^{-\tfrac\delta6 \left(\sum_{k\geq0}\frac{1}{2k}\right)\beta} \\
&\geq \ee^{(\Theta-\Gamma)\beta} \ee^{\theta\beta} \ee^{-\tfrac\delta3\beta}.
\end{array}
\]


\subsection{Exclude escape via coalescence}
\label{sec:nocoal}

In order to prove Theorem \ref{thm:tube} for $\Theta\leq\theta$, we first need to exclude the escape from $\cR$ by coalescence. 

We say that {\it the exit from $\cR$ occurs via coalescence} if, starting from a configuration with all the clusters having at most $\ell_c(\ell_c-1)+1$ particles each, a cluster with at least $\ell_c(\ell_c-1)+2$ particles is obtained after the move of one particle with the attachment of at least two {\it real} clusters (i.e., with at least two particles). We prove that
\begin{equation}
\label{eq:coal}
\begin{array}{ll}
&P_{\mu_{\cal R}} \left( \hbox{the exit from } \cR \hbox{ occurs via coalescence} \right) \\ &\quad\qquad\qquad\qquad\leq \SES + P_{\mu_{\cal R}} \left(\tau_{\cR^c}\leq \ee^{(\Gamma-\Theta+\delta)\beta}
\hbox{ and the exit occurs via coalescence} \right) \\
&\quad\qquad\qquad\qquad\leq \ee^{-(\Delta-U)\beta}e^{3\delta\beta}.
\end{array}
\end{equation}
Using the notations introduced in Section \ref{sec:lowbound}, we note that, if $X$ exits from $\cR$ within time $t$, then there exists $0\leq s\leq t$ such that $X(s)\in\bigcup_{k<4}\bigcup_{i<n} (B_i^k \cup C_i^k)$, where $n=|\Lambda_\beta|/(2\ell_c)^2$, 
\[
\begin{aligned}
B_i^k &=\{ \eta\in\cR\colon\, |\eta_i^{k,\mathrm{in}}|\geq\ell_c(\ell_c-1)+3 \hbox{ and } \eta_i^{k,\mathrm{in}}\\ 
&\qquad \hbox{ has no cluster with more than } \ell_c(\ell_c-1)+1 \hbox{ particles} \}
\end{aligned}
\]
and
\[
\begin{array}{ll}
C_i^k=\{ \eta\in\cR\colon\, |\eta_i^{k,\mathrm{in}}|\geq\ell_c(\ell_c-1)+2, \, \eta_i^{k,\mathrm{in}} \hbox{ has no cluster with more than } \ell_c(\ell_c-1) \hbox{ particles}, \\
\qquad\quad \hbox{and it does not consist of only two quasi-squares of dimensions } 1\times2 \hbox{ and } (\ell_c-1)\times\ell_c \}.
\end{array}
\]
Indeed, note that the exit from $\cR$ via coalescence implies that either a cluster with $\ell_c(\ell_c-1)+1$ particles is involved in the move, or a cluster with $\ell_c(\ell_c-1)$ particles at most is. In the former case, since two real clusters need to attach, $|\eta_i^{k,\mathrm{in}}|\geq\ell_c(\ell_c-1)+3$ and therefore $\eta\in B_i^k$. In the latter case note that $\eta\in C_i^k$, where the subcase in which $\eta_i^{k,\mathrm{in}}$ contains only two quasi-squares of dimensions $1\times2$ and $(\ell_c-1)\times\ell_c$ is excluded, because it is geometrically impossible to attach two real clusters with a single move without the addition of another particle.

It is now enough to show that, for $\eta^{\mathrm{in}}\in B_i^{k,\mathrm{in}} \cup C_i^{k,\mathrm{in}}$, we have $H(\eta^{\mathrm{in}})+\Delta|\eta^{\mathrm{in}}|\geq \Gamma + (\Delta-U)$. This is straightforward in the case $\eta^{\mathrm{in}}\in B_i^{k,\mathrm{in}}$. When $\eta^{\mathrm{in}}\in C_i^{k,\mathrm{in}}\setminus B_i^{k,\mathrm{in}}$, so that $|\eta_i^{k,\mathrm{in}}|=\ell_c(\ell_c-1)+2$, we define the configuration $\bar\eta^{\mathrm{in}}$ as the one obtained after replacing each cluster in $\eta^{\mathrm{in}}$ by the smallest quasi-square having volume larger or equal than that of the cluster, in such a way the quasi-squares do not intersect. Note that $H(\eta^{\mathrm{in}})\geq H(\bar\eta^{\mathrm{in}})$ and $\bar\eta^{\mathrm{in}}$ contains subcritical quasi-squares only, i.e., the largest can be a $(\ell_c-1)\times\ell_c$ quasi-square. Let us denote by $Q_1,\ldots,Q_m$ the real quasi-squares in $\bar\eta^{\mathrm{in}}$, with $m\geq2$, and set $M=|Q_1|+\ldots+|Q_m|$. By using \cite[Corollary 6.18]{CN13}, we know that the perimeter of a cluster with volume $K$ is bounded from below by $4\sqrt{K}$. Thus, we deduce that the total perimeter of the $m$ disjoint clusters is at least $4(\sqrt{|Q_1|}+\ldots+\sqrt{|Q_m|})$. By direct computations we deduce that the minimum value of the perimeter is achieved, up to any permutation, for $|Q_1|=\ldots=|Q_{m-1}|=2$ and $|Q_m|=M-2(m-1)$, when $Q_1,\ldots,Q_{m-1}$ are a $1\times2$ quasi-square and $Q_m$ is a quasi-square with $M-2(m-1)$ particles. For $m=2$, note that $|Q_1|=2$ and $|Q_2|=\ell_c(\ell_c-1)$. If $\bar\eta^{\mathrm{in}}$ contains at least one free particle, then $H(\bar\eta^{\mathrm{in}})+\Delta|\bar\eta^{\mathrm{in}}|\geq\Gamma+\Delta>\Gamma+(\Delta-U)$. Otherwise, since the configurations in $C_i^k$ do not consist of only two quasi-squares of dimensions $1\times2$ and $(\ell_c-1)\times\ell_c$ even when $\bar\eta^{\mathrm{in}}$ does, we get that $H(\eta^{\mathrm{in}})+\Delta|\eta^{\mathrm{in}}|\geq \Gamma+U>\Gamma+(\Delta-U)$. For $m\geq3$, a similar argument applies. This concludes the proof of \eqref{eq:coal} after arguing as in Section \ref{sec:lowbound}.


\subsection{Subcritical tube of typical trajectories}
\label{sec:subtube}

Here we describe the uphill motion of the dynamics. From the typical transition theorem for subcritical quasi-squares (\cite[Theorem 1.8]{BGdHNOS23}), we know that a $\ell_1\times\ell_2$ quasi-square, with $2\leq\ell_1<\ell_c$, decays into a $(\ell_2-1)\times\ell_1$ quasi-square,
while a $2\times2$ square dissolves, up to time $T^{\star}$ defined in \eqref{def:Tstella}. To construct the subcritical part of the tube of typical trajectories, we use the argument carried out in \cite{S92} exploiting the reversibility of the dynamics. In particular, the first part of the typical tube of exiting trajectories is the time-reversal of a typical evolution of a shrinking subcritical quasi-square.  Since we have excluded the exit from $\cR$ via coalescence, we conclude that, with probability exponentially close to $1$ the system nucleates a cluster that visits the full increasing sequence of quasi-squares up to the $\ell_c\times\ell_c$ square. This also gives $\tau_{(2,2)}^{last}>\tau_{\cR^c}-\ee^{(\theta+\delta)\beta}$ with probability exponentially close to 1.


\subsection{Supercritical tube of typical trajectories}
\label{sec:suptube}

Here we describe the downhill motion of the dynamics. Using the subcritical part of the typical trajectories described in Section \ref{sec:subtube},
we know that with probability exponentially close to $1$ the dynamics nucleates a wandering cluster that visits all the quasi-squares up to the $\ell_c\times\ell_c$ square. Afterwards, by the typical transition theorem for supercritical quasi-squares (\cite[Theorem 1.8]{BGdHNOS23}), we know that the cluster grows in the same way up to time $T^{\star}$ and within time $\ee^{(2\Delta-U)\beta}\ee^{\delta\beta}$. By Theorem \ref{thm:escape} and the construction of the tube of typical trajectories given above, the time $T^{\star}$ coincides with the appearance time of a single large cluster of volume $\frac{1}{8}\lambda(\beta)$ with probability tending to $1$. Thus, our description of the supercritical part of the tube of typical trajectories holds up to the formation of a  $\sqrt{\lambda(\beta)/8}\times\sqrt{\lambda(\beta)/8}$ square.


\section{Upper bound and tube: large volume}
\label{sec:upplarge}

In this section we prove the upper bound for the time $\tau_{\cR^c}$ in Theorem \ref{thm:escape} and identify the nucleation pattern stated in Theorem \ref{thm:tube}, in the case of large volume, i.e., $\theta<\Theta<\Gamma-(2\Delta-U)$. In Section~\ref{sec:ind} we first construct a toy model for our droplet dynamics on moderately large boxes and identify the exit time from $\cR$ for this model. In Section~\ref{sec:nearind} we use the latter, together with the near-independence of the dynamics in different boxes of moderate volume, to estimate the nucleation time and identify the nucleation pattern.


\subsection{Independence of moderately large boxes: a toy model}
\label{sec:ind}

We can view a simple birth-death chain $\xi$ on 
	\[
	\hbox{QS}_0=\{(\ell_1,\ell_2)\in\hbox{QS}:2\leq\ell_1\leq\ell_2\leq\ell_c\}
	\]
as a {\it caricature} of the original process until the exit from $\cR$ in the case $\Theta\leq\theta$. Letting $s=(\ell_1,\ell_2)$, we set
\[
s+1=\left\{
\begin{array}{ll}
	(\ell_2,\ell_1+1), &\hbox{if } 2\leq\ell_1\leq\ell_2\leq\ell_1+1\leq\ell_c, \\
	(2,2), &\hbox{if } s=(0,0),
\end{array}
\right.
\]
and
\[
s-1=\left\{
\begin{array}{ll}
	(\ell_2-1,\ell_1), &\hbox{if } 2<\ell_2\leq\ell_1+1\leq\ell_c, \\
	(0,0), &\hbox{if } s=(2,2).
\end{array}
\right.
\]
Thus, we define $\xi$ as a birth-death process with transition probability
\[
p(s,s+1)=\left\{
\begin{array}{ll}
	\ee^{-(r(0,0)-\Delta)\beta}, &\hbox{if } s=(0,0), \\
	\ee^{-(\Delta-U)\beta}, &\hbox{if } s\neq(0,0),
\end{array}
\right.
\]
where $r(0,0)=4\Delta-2U-\Theta$ is the resistance of a configuration in $\cX_E$, and
\[
p(s,s-1)=\ee^{-(r(s)-\Delta)\beta}, \quad \hbox{if } s\neq(0,0),
\]
where $r(s)$ is defined in \eqref{def:res}. Set $p((\ell_c,\ell_c),(\ell_c,\ell_c))=1$. The process $\xi$ can be seen as a caricature of the original process because, when we look the process $X$ at the return times to $\cX_{\Delta^+}$ and keep track of the dimensions of the smallest quasi-square only, then $\xi$ does not take care of coalescent events and long jumps from the state $(2,2)$, but does keep the right scale of the exit time and the description of the subcritical tube of typical trajectories. In the case $\Theta>\theta$, we define another caricature $\zeta$, for which we will do our computations, before identifying the key properties that lead to the desired result. In Section \ref{sec:nearind} we show that the original model satisfies these properties, thereby obtaining Theorems \ref{thm:escape} and \ref{thm:tube}.
	
When going from $\Theta\leq\theta$ to $\Theta<\Gamma-(2\Delta-U)$, we face two difficulties:
\begin{itemize}
	\item There is no simple recurrence property to $\cX_{\Delta^+}$ on a time scale of order $T_{\Delta}$. By dividing $\Lambda_\beta$ into boxes of volume $\ee^{\theta\beta}$, we could show such a recurrence for the restrictions of $X$ to each box, but the associated return times are not simultaneous.
	\item Even when on the time scale $T_\Delta$ the dynamics on such fixed boxes of volume $\ee^{\theta\beta}$ are essentially independent, the correlations grow on longer time scales. To overcome this difficulty, we need to dynamically redefine volumes of order $\ee^{\theta\beta}$, which we will refer to as {\it clouds}, when we observe the starting of a nucleation event only. In the caricature $\xi$ such a starting event is represented by a jump from $(0,0)$ to $(2,2)$. The caricature $\zeta$ will be instead defined as an {\it aggregation} of birth-death chains on $\hbox{QS}_0$ absorbed in $(0,0)$ and $(\ell_c,\ell_c)$. We can see an absorption in $(0,0)$ as a failed nucleation attempt, an absorption in $(\ell_c,\ell_c)$ as a successful attempt, and the arrival of a new chain as the starting of such an attempt.
\end{itemize}


\paragraph{Toy model $\zeta$.} 

$\zeta$ is a Markov chain on $\cup_{k\geq0}\hbox{QS}_0^k$. If $\zeta_n\in\hbox{QS}_0^k$, then we write $|\zeta_n|=k$. If $\zeta_n=(\xi_{i,n})_{i<k}$, with $k=|\zeta_n|$, then we have $|\zeta_{n+1}|=|\zeta_n|+A$, where $A$ is a Poisson random variable with mean $a(\beta)=\ee^{\Theta\beta}\ee^{-(4\Delta-2U)\beta}\ee^{\Delta\beta}$. For $k\leq i<k+A$, $\xi_{i,n+1}=(2,2)$. For $i<k$, $\xi_{i,n+1}$ is sampled according to $p(\xi_{i,n},\cdot)$ if $\xi_{i,n}\neq(0,0)$, while $\xi_{i,n+1}=(0,0)$ if $\xi_{i,n}=(0,0)$.
	
Note that, in the case $\Theta=\theta$, we have $a(\beta)=p((0,0),(2,2))$, while, in the case $\Theta>\theta$, $a(\beta)$ can be viewed as the expected nucleation rate of a $2\times2$ quasi-square anywhere in $\Lambda_\beta$ within time $T_\Delta$. We will study $\zeta$ up to the first absorbing time of one of its coordinates $\xi$, which we will refer to as {\it histories}, in the state $(\ell_c,\ell_c)$.

Starting from a list of dimensions $(\ell_{1,i},\ell_{2,i})_{i < k}$ with $2\leq\ell_{1,i}\leq\ell_{2,i}\leq\ell_{1,i}+1\leq\ell_c$ for any $i<k$, each history $\xi_i$ typically fails, i.e., has only a small probability to reach $(\ell_c,\ell_c)$. Each history $\xi_i$ of the process $\zeta$ is absorbed within time $\ee^{(\theta-\Delta)\beta}\ee^{\delta\beta}$ with probability $1-\SES$. Furthermore, the probability of being absorbed in $(\ell_c,\ell_c)$ within time $\ee^{(\theta-\Delta)\beta}\ee^{\delta\beta}$ is at least $\ee^{-(\Gamma-(4\Delta-2U))\beta}$. Indeed, 
\begin{equation}
\label{eq:absorption}
\begin{array}{ll}
&P_{(2,2)}\left(\tau_{(\ell_c,\ell_c)}<\tau_{(0,0)}, \tau_{(\ell_c,\ell_c)}<\ee^{(\theta-\Delta+\delta)\beta}\right) \\[0.2cm]
&\qquad = P_{(2,2)}\left(\tau_{(\ell_c,\ell_c)}<\tau_{(0,0)}, 
\tau_{(0,0)} \wedge \tau_{(\ell_c,\ell_c)}<\ee^{(\theta-\Delta+\delta)\beta}\right) \\[0.2cm]
&\qquad = P_{(2,2)}\left(\tau_{(\ell_c,\ell_c)}<\tau_{(0,0)}\right) 
- P_{(2,2)}\left(\tau_{(\ell_c,\ell_c)}<\tau_{(0,0)}, 
\tau_{(0,0)} \wedge \tau_{(\ell_c,\ell_c)}\geq\ee^{(\theta-\Delta+\delta)\beta}\right) \\[0.2cm]
&\qquad \geq \ee^{-(\Gamma-(4\Delta-2U))\beta}-\SES.
\end{array}
\end{equation}
Within $k$ steps of $\zeta$, with probability $1-\SES$ we have that at least $k a(\beta)\ee^{-\frac{\delta}{4}\beta}$ new histories $\xi_i$ are born. Thus, for $k\gg\ee^{(\theta-\Delta)\beta}$, the probability that no such chains $\xi_i$ manage to reach $(\ell_c,\ell_c)$ before failing is at most 
\[
\left(1-\ee^{-(\Gamma-(4\Delta-2U))\beta}\right)^{k a(\beta)\ee^{-(\delta/4)\beta}}=\SES 
\quad \hbox{if} \quad k\gg\ee^{(\Gamma-\Theta-\Delta)\beta}.
\]
Thus, the exit time from the analogous set $\cR$ has the order claimed in Theorem \ref{thm:escape} for this toy model, thanks to \cite[Theorem 1.5]{BGdHNOS23}. The three key properties of this toy model are the following:
\begin{enumerate}
\item[(1)] The evolutions of the histories $\xi_i$ are independent.
\item[(2)] For any $\delta>0$, the probability that a chain $\xi_i$ reaches $(\ell_c,\ell_c)$ within time $\ee^{(\theta-\Delta)\beta}\ee^{\delta\beta}$ is at least $\ee^{-(\Gamma-(4\Delta-2U))\beta}\ee^{\delta\beta}$.
\item[(3)] At each step a Poisson number with mean $a(\beta)$ of histories are started.
\end{enumerate}

In Section \ref{sec:nearind} we will show that for our original model we can define an aggregation of histories that is {\it close} to the toy model, i.e., up to negligible events properties (1)-(3) still hold. This will imply Theorem \ref{thm:escape}.

Theorem \ref{thm:tube} follows after using the same tools. As in Section \ref{sec:nocoal}, the absence of exit from $\cR$ via coalescence is a corollary of Theorem \ref{thm:escape}. We can again follow the argument in \cite{S92} to construct the subcritical tube of typical trajectories as in Section \ref{sec:subtube}. In this context we just need the analogue of property (1). With this property, the control of the supercritical tube of typical trajectories is obtained in the same manner.


\begin{figure}
	\centering
	\begin{tikzpicture}[scale=0.15,transform shape]
		\draw[fill=gray!40] (2,2) rectangle (3,3);
		\draw[fill=gray!40] (3,2) rectangle (4,3);
		\draw[fill=gray!40] (4,2) rectangle (5,3);
		\draw[fill=gray!40] (5,2) rectangle (6,3);
		\draw (6,2) rectangle (7,3);
		\draw (7,2) rectangle (8,3);
		\draw[fill=gray!40] (2,3) rectangle (3,4);
		\draw[fill=gray!40] (3,3) rectangle (4,4);
		\draw[fill=gray!40] (4,3) rectangle (5,4);
		\draw[fill=gray!40] (5,3) rectangle (6,4);
		\draw[fill=gray!40] (2,4) rectangle (3,5);
		\draw[fill=gray!40] (3,4) rectangle (4,5);
		\draw[fill=gray!40] (4,4) rectangle (5,5);
		\draw[fill=gray!40] (5,4) rectangle (6,5);
		\draw[fill=gray!40] (17,6) rectangle (18,7);
		\draw[fill=gray!40] (18,6) rectangle (19,7);
		\draw[fill=gray!40] (19,6) rectangle (20,7);
		\draw[fill=gray!40] (17,7) rectangle (18,8);
		\draw[fill=gray!40] (18,7) rectangle (19,8);
		\draw[fill=gray!40] (19,7) rectangle (20,8);
		\draw[fill=gray!40] (17,8) rectangle (18,9);
		\draw[fill=gray!40] (18,8) rectangle (19,9);
		\draw[fill=gray!40] (19,8) rectangle (20,9);
		\draw[fill=gray!40] (11,-6) rectangle (12,-5);
		\draw[fill=gray!40] (12,-6) rectangle (13,-5);
		\draw[fill=gray!40] (11,-5) rectangle (12,-4);
		\draw[fill=gray!40] (12,-5) rectangle (13,-4);
		\draw (12,0) rectangle (13,1);
		\draw (13,0) rectangle (14,1);
		\draw (7,9) rectangle (8,10);	
		\draw (40,5) rectangle (50,15);
	\draw[fill=gray!40] (42,7) rectangle (43,8);
	\draw[fill=gray!40] (43,7) rectangle (44,8);
	\draw[fill=gray!40] (44,7) rectangle (45,8);
	\draw[fill=gray!40] (45,7) rectangle (46,8);
	\draw[fill=gray!40] (46,7) rectangle (47,8);
	\draw[fill=gray!40] (42,8) rectangle (43,9);
	\draw[fill=gray!40] (43,8) rectangle (44,9);
	\draw[fill=gray!40] (44,8) rectangle (45,9);
	\draw[fill=gray!40] (45,8) rectangle (46,9);
	\draw[fill=gray!40] (46,8) rectangle (47,9);
	\draw[fill=gray!40] (42,9) rectangle (43,10);
	\draw[fill=gray!40] (43,9) rectangle (44,10);
	\draw[fill=gray!40] (44,9) rectangle (45,10);
	\draw[fill=gray!40] (45,9) rectangle (46,10);
	\draw[fill=gray!40] (46,9) rectangle (47,10);
	\draw (43,10) rectangle (44,11);
	\draw (43,11) rectangle (44,12);
	\draw (45,12) rectangle (46,13);
	
	\draw (34,0) rectangle (35,1);
	\draw (35,0) rectangle (36,1);
	
	\draw (30,-12) rectangle (31,-11);
	\draw (30,-11) rectangle (31,-10);
	\draw (30,-10) rectangle (31,-9);
	
		\draw (-1,-9) rectangle (23,14);
		\put(0,65){\large $C_0$};		
		
		
		\draw[fill=gray!40] (53,-12) rectangle (54,-11);
		\draw[fill=gray!40] (54,-12) rectangle (55,-11);
		\draw[fill=gray!40] (55,-12) rectangle (56,-11);
		\draw[fill=gray!40] (56,-12) rectangle (57,-11);
		\draw[fill=gray!40] (53,-11) rectangle (54,-10);
		\draw[fill=gray!40] (54,-11) rectangle (55,-10);
		\draw[fill=gray!40] (55,-11) rectangle (56,-10);
		\draw[fill=gray!40] (56,-11) rectangle (57,-10);
		\draw (53,-10) rectangle (54,-9);
		\draw (50,-15) rectangle (60,-6);
		
		\put(175,70){\large $C_1$};
		\put(215,-20){\large $C_2$};
		
	\end{tikzpicture}
	\caption{\small Construction of the clouds $(C_i)_{i<3}$ at time $t=0$. Note that $\dist(C_{i_1},C_{i_2})=\sqrt{T_{\theta}\ee^{-\kappa\beta}}$, with $\kappa>0$ small enough, for $i_1,i_2<3$ and $i_1\neq i_2$. 
		\normalsize}
	\label{fig:cloud}
\end{figure}
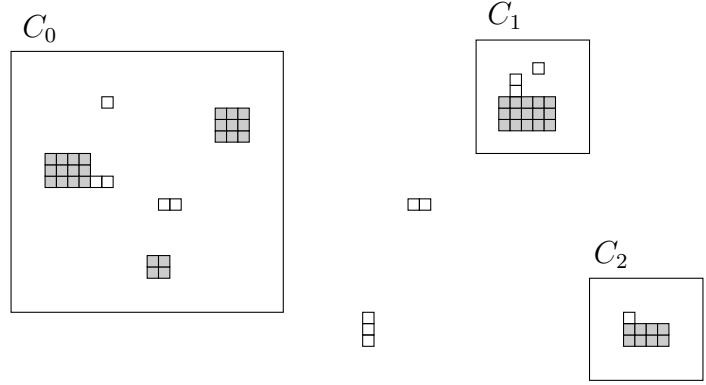


\subsection{Near-independence of moderately large boxes for original model}
\label{sec:nearind}

At time $t=0$, we declare as sleeping all the particles that belong to a quasi-square $(\ell_1,\ell_2)$ with $\ell_1\geq2$. We define $A(0)$ as the set of sites in $\Lambda_\beta$ that are occupied by a sleeping particle, and put, after recalling \eqref{eq:orlo}, for small enough $\kappa>0$
\begin{equation}\label{eq:cloud1}
\bar{S}(0) = [A(0),\theta-\kappa].
\end{equation}
At time $t=0$, we define the set of {\it clouds} $(C_i)_{i < N(0)}$ as $g_{r_0}(\bar S(0))$, which is a collection of rectangles at distance at least 
\[
r_0=\ee^{\frac{1}{2}(\theta-\kappa)\beta}
\]
from each other. (See Fig.~\ref{fig:cloud}.) Let $X_i$ be the restriction of the process $X$ to the cloud $C_i$. We couple $X_i$ with $X_i^{pb}$, which is a Kawasaki dynamics on $C_i$ with periodic boundary conditions and initial condition $X_i(0)$. To construct this coupling we assign colours to some particles, as follows. We paint {\it black} the particles of the process $X_i$ that visit $\Lambda_\beta\setminus C_i$ or share a cluster with a black particle. We paint {\it white} the particles of the process $X_i^{pb}$  that visit the internal boundary of $C_i$ or share a cluster with a white particle. All the other particles have no colour, both in $X_i$ and $X_i^{pb}$. Let $B(t)$ and $W(t)$ be the sets of sites in $C_i$ that were visited by a black and a white particle, respectively, within time $t$. \cite[Theorem 2 and Corollary 1.2]{GN10} provide a control on the growth of these two sets. Indeed, for any $t\geq \ee^{\Delta\beta}$ and $\delta>0$, with probability $1-\SES$ the two zones remain contained within a distance $t\ee^{\Delta\beta}\ee^{\delta\beta}$ from the boundary of $C_i$. We use the natural coupling that ensures equality of $X_i$ and $X_i^{pb}$ in $C_i\setminus(B(t)\cup W(t))$ for any $t\geq0$. With any $X_i^{pb}$ we associate the sequence $(\bar\tau_{k,i})_{k\geq0}$ of return times in $\cX_{\Delta^+}$ after seeing an active particle in the local boxes (see \cite[Eq.\ (1.22)-(1.23)]{BGdHNOS23}). We define the stopping times
\[
\tau_0^* = \displaystyle\max_{i < N(0)} \bar\tau_{0,i}
\]
and
\[
\tau_{0,i}^* = \displaystyle\max_{k\geq 0}\{\bar\tau_{k,i}\leq\tau_0^*\},
\]
i.e., the last time before $\tau_0^*$ when $X_i^{pb}$ recurs in $\cX_{\Delta^+}$. 

\begin{figure}
	\centering
	\begin{tikzpicture}[scale=0.35,transform shape]
		
		\draw[|->] (0,32) -- (30,32);
		\put(-25,315){\Large $\Xi_0$};
		\put(-3,307){$0$};
		\put(14,314.5){{\LARGE $\times$}};
		\draw (2.05,32) circle (0.53);
		\put(41,314.5){\LARGE $\times$};
		\put(63,314.5){\LARGE $\times$};
		\draw (6.95,32) circle (0.53);
		\put(120,314.5){\LARGE $\times$};
		\draw (12.7,32) circle (0.53);
		\draw (12.16,31.46) rectangle (13.22,32.52); 
		\draw (12.7,34) -- (12.7,18.2);
		\put(220,314.5){\LARGE $\times$};
		\draw (22.7,32) circle (0.53);
		\draw (22.2,31.46) rectangle (23.26,32.52); 
		\draw (22.75,34) -- (22.75,18.2);
				
		\draw[|->] (0,29) -- (30,29);
		\put(-25,285){\Large $\Xi_1$};
		\put(-3,277){$0$};
		\put(27,284.5){\LARGE $\times$};
		\draw (3.35,29) circle (0.53);
		\draw (2.82,28.46) rectangle (3.88,29.52); 
		
		\draw (3.4,34) -- (3.4,18.2);
		\put(28,170){$\tau_0^*$};
		\put(80,170){$\tau_1^*$};
		\put(122,170){$\tau_2^*$};
		\put(222,170){$\tau_3^*$};
		
		\put(52,284.5){\LARGE $\times$};
		\draw (5.87,29) circle (0.53);
		\put(83,284.5){\LARGE $\times$};
		\put(97,284.5){\LARGE $\times$};
		\draw (10.37,29) circle (0.53);
		\put(133,284.5){\LARGE $\times$};
		\put(153,284.5){\LARGE $\times$};
		\put(166,284.5){\LARGE $\times$};
		\put(203,284.5){\LARGE $\times$};
		\draw (21.07,29) circle (0.53);
		\put(250,284.5){\LARGE $\times$};
				
		\draw[|->] (0,26) -- (30,26);
		\put(-25,255){\Large $\Xi_2$};
		\put(-3,247){$0$};
		\put(18,254.5){\LARGE $\times$};
		\draw (2.45,26) circle (0.53);
		\put(78,254.5){\LARGE $\times$};
		\draw (8.5,26) circle (0.53);
		\draw (8,25.46) rectangle (9,26.52); 
		
		\draw (8.5,34) -- (8.5,18.2);
		
		\put(93,254.5){\LARGE $\times$};
		\put(109,254.5){\LARGE $\times$};
		\draw (11.57,26) circle (0.53);
		\put(145,254.5){\LARGE $\times$};
		\put(175,254.5){\LARGE $\times$};
		\put(210,254.5){\LARGE $\times$};
		\draw (21.72,26) circle (0.53);

		\draw[|->] (8.5,23) -- (30,23);
		\put(-25,225){\Large $\Xi_3$};
		\put(105,225){\LARGE $\times$};
		\draw (11.2,23) circle (0.53);
		\put(135,225){\LARGE $\times$};
		\put(175,225){\LARGE $\times$};
		\put(205,225){\LARGE $\times$};
		\draw (21.25,23) circle (0.53);
		\put(235,225){\LARGE $\times$};
		\put(260,225){\LARGE $\times$};

		\draw[|->] (8.5,20) -- (30,20);
		\put(-25,195){\Large $\Xi_4$};
		\put(97,195){\LARGE $\times$};
		\draw (10.35,20) circle (0.53);
		\put(127,195){\LARGE $\times$};
		\put(147,195){\LARGE $\times$};
		\put(158,195){\LARGE $\times$};
		\put(193,195){\LARGE $\times$};
		\draw (20,20) circle (0.53);
		
	\end{tikzpicture}
	\vskip 0.1cm
	\caption{\small We depict the construction of the times $(\tau^*_{j,i})_{j < 4}$ for $i<N(j)$, where $N(0)=N(1)=3$ and $N(2)=N(3)=5$. In particular, each cross represents a return time $\bar\tau_{j,i}$, where we highlight with a circle the times $\tau^*_{j,i}$ and with a square the times $\tau_j^*$. \normalsize}
	\label{fig:orizzonte}
\end{figure}
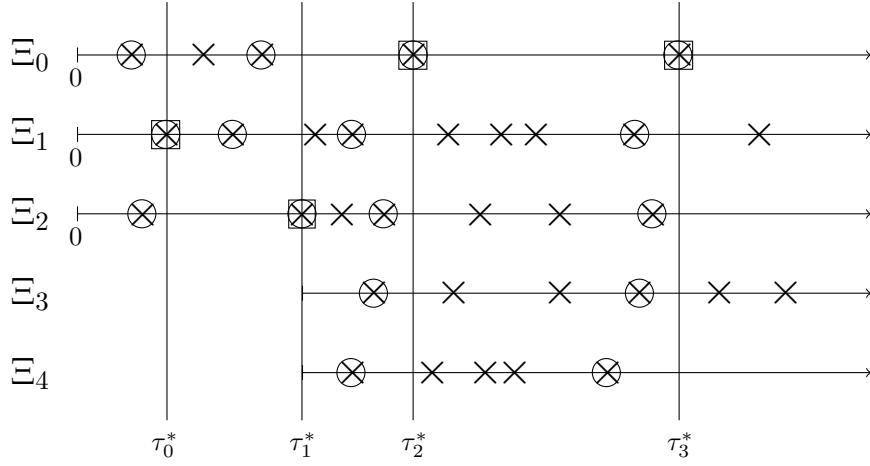

At time $t=\tau^*_{j-1}$, $j\geq1$, we iteratively construct the set of clouds $(C_i)_{i < N(j)}$ in a similar manner, where we now replace $r_0$ by
\[
r_j=\ee^{-\frac{1}{2}\frac{\kappa}{2^j}\beta}r_{j-1},
\]
i.e., replace $\theta-\kappa$ in \eqref{eq:cloud1} by $\theta-\sum_{i=0}^j \kappa/2^j$.
We then define the stopping times 
\[
\tau_j^* = \displaystyle\max_{i < N(j)}\displaystyle\min_{k \geq 0} \{\bar\tau_{k,i}\colon\, \bar\tau_{k,i} > \tau_{j-1}^*\}, 
\qquad j \geq 1,
\]
i.e., the first time after time $\tau_{j-1}^*$ when all the new $X_i^{pb}$ have recurred in $\cX_{\Delta^+}$, with initial conditions at time $\tau_{j-1}^*$. We also define
\[
\tau_{j,i}^* = \max_{k\geq0}\{\bar\tau_{k,i}\leq\tau_j^*\}.
\]
We refer the reader to Fig.~\ref{fig:orizzonte}, where the trajectories on each time interval $[\tau_{j-1}^*,\tau_j^*]$ have been reorganised according to histories $\Xi_i$, $i\geq0$, that we define below. Note that $\tau^*_{j,i}>\tau^*_{j-1}$ for any $j,i$. We then define the process
\[
Z_j=(X_i(\tau^*_{j,i}))_{i < N(j)}, \qquad 0\leq j\leq J^\star,
\]
where 
\[
J^\star=\max_{j\geq0}\{\tau^*_j\leq T^\star\}.
\]
It is important that we monitor the droplet dynamics at times $(\tau^*_{j,i})_{i < N(j)}$, $j\geq0$, rather than at times $(\bar\tau_{k,i})_{k\geq0, i < N(j)}$, $j\geq0$, because for the former all the clouds {\it recur almost simultaneously} to $\cX_{\Delta^+}$, i.e., within a time interval of length $T_{\Delta^+}$. This fact is needed to obtain the near independence of the droplet dynamics in these clouds. As we view $\zeta$ as an aggregation of independent histories $\xi_i$, which can either fail or succeed, we view $Z$ as a collection of {\it essentially} independent and {\it active} histories $\Xi_i$ of configurations in $\cX_{\Delta^+}$. We say that a history is active when it has not failed. $Z$ forgets about all failed histories. To identify a history $\Xi_i$ along the configurations $Z_0=(X_k(\tau^*_{0,k}))_{k < N(0)}$, $Z_1=(X_k(\tau^*_{1,k}))_{k < N(1)}, \ldots$, we have to
\begin{itemize}
\item[(i)] 
decide which is the successor $\Xi_{i,n+1}$ in $Z_{n+1}$ of $\Xi_{i,n}$ in $Z_n$ for any $n\geq0$, if such a successor exists;
\item[(ii)] 
identify the death, or failure, events of an history $\Xi_i$, i.e., the absence of a successor $\Xi_{i,n+1}$ in $Z_{n+1}$ of $\Xi_{i,n}$ in $Z_n$;
\item[(iii)] 
associate with the appearance of each $X_k(\tau_{n+1,k}^*)$ in $Z_{n+1}$ without predecessor in $Z_n$ the start of a new history $\Xi_i$.
\end{itemize}
To this end, in each configuration $\Xi_{i,n}=X_k(\tau_{n,k}^*)$ in $Z_n$ we consider one of the clusters containing at least one sleeping particle outside the black and the white zones (where the clusters coincide with the quasi-squares of $X_k^{pb}$), and we choose the one with the smallest dimensions. In this cluster we consider the particle that is asleep for a longer time (or one of these, if they are more than one). As long as this particle is still sleeping in one of the $X_{k'}(\tau_{n+1,k'}^*)$ in $Z_{n+1}$, we set $\Xi_{i,n+1}=X_{k'}(\tau_{n+1,k'}^*)$. Otherwise, we say that $\Xi_{i,n}$ has no successor, which corresponds to a death and failure of the history $\Xi_i$. We then associate with each $X_k(\tau_{n+1,k}^*)$ in $Z_{n+1}$ without predecessor the start of a new history $\Xi_i$.

\begin{remark}
{\rm The choice that $(r_k)_{k\geq0}$ is a decreasing sequence allows us to control the appearance of clusters outside or at the boundary of the clouds at the previous step without having an increasing number of clusters along a history $\Xi_i$. (It may however happen that the number of clusters decreases together with a birth event: this occurs whenever the clusters in the same cloud become clusters in different clouds at the next time step.) Then the number of clusters can increase along a history $\Xi_i$, but up to coalescence and with a negligible probability compared to that of growing a row on the smallest quasi-square.}
\end{remark}

\begin{remark}
\label{rmk:grow}
{\rm From \cite[Theorem 1.5]{BGdHNOS23}, up to an event of probability $\SES$, there are no more than $\ee^{(\alpha+\delta)\beta}$ return times in $\cX_{\Delta^+}$ of $X_k^{pb}$ between two successive configurations of the same history $\Xi_i$ for any $\delta>0$. This, together with \cite[Theorems 1.6, 1.8 and 1.9]{BGdHNOS23}, implies that, while in $\cR$ and unless a coalescence occurs, the smallest quasi-square $(\ell_1,\ell_2)$ along a history $\Xi_i$
\begin{itemize}
\item[(i)] typically resists for a time of order $\ee^{(r(\ell_1,\ell_2)-\Delta)\beta}$ if $\ell_1\geq2$;
\item[(ii)] typically either decays in a quasi-square $(\ell_2-1,\ell_1)$ if $\ell_2>2$ or dies if $(\ell_1,\ell_2)=(2,2)$;
\item[(iii)] grows into a quasi-square $(\ell_2,\ell_1+1)$ with a probability of order $\ee^{-((2\Delta-U)-\Delta)\beta}$.
\end{itemize}
This shows that each history typically dies or exits from $\cR$ within a time of order $\ee^{(\theta-\Delta)\beta}$ with probability $1-\SES$.}
\end{remark}

The proposition below is the core result of this section and ensures that the key properties of the toy model $\zeta$ also hold for our aggregation of histories $\Xi_i$.

\begin{proposition}
\label{prp:properties}
Up to an event of probability $\SES$:
\begin{enumerate}
\item[(1)] The evolutions of the histories $\Xi_i$ are independent.
\item[(2)] The probability that a history $\Xi_i$ reaches $(\ell_c,\ell_c)$ before dying within time $\ee^{(\theta-\Delta)\beta}\ee^{\delta\beta}$ is at least $\ee^{-(\Gamma-(4\Delta-2U))\beta}\ee^{\delta\beta}$ for any $\delta>0$.
\item[(3)] At each step at least a Poisson number with mean $a(\beta)$ of new histories start.
\end{enumerate}
\end{proposition}

\begin{proof}
Property (1) is immediate from \cite[Theorem 3]{GN10}. Indeed, at each step $0\leq j\leq J^\star$, uniformly in $i_1,i_2 < N(j)$, the distance between the clouds $C_{i_1}$ and $C_{i_2}$ is at least $r_j\gg\ee^{\frac{\Delta}{2}\beta}\ee^{\delta\beta}$ for any $\delta>0$, which implies that the condition in \cite[Eq.\ (1.23)]{GN10} is satisfied. Property (2) readily follows from Remark \ref{rmk:grow} after arguing as in \eqref{eq:absorption}. Property (3) follows from the fact that the number of new clouds that are created at each step dominates a Poisson random variable with mean $a(\beta)=\ee^{\Theta\beta}\ee^{-(4\Delta-2U)\beta}\ee^{\Delta\beta}$. This is related to creation of new local boxes and can be obtained after arguing as in the proof of \cite[Lemma 4.1]{BGdHNOS23}.
\end{proof}

By Proposition \ref{prp:properties}, it is immediate that, for any $\delta>0$,
\[
P_{\mu_{\cR}}\left(\tau_{\cR^c}>\dfrac{\ee^{\Gamma\beta}}{|\Lambda_\beta|}\ee^{\delta\beta}\right)
\leq\left(1-\ee^{-(\Gamma-(4\Delta-2U))\beta}\ee^{-\delta\beta}\right)^{\ee^{(\Gamma-(3\Delta-2U))\beta}
\ee^{-\frac{\delta}{4}\beta}}=\SES.
\]
This, together with the argument developed in Section \ref{sec:lowbound}, implies Theorem \ref{thm:escape}. Once we provide the upper bound for the exit time from $\cR$, we can use the same argument as in Section \ref{sec:nocoal} to exclude the escape via coalescence in the case of large volumes, so that in this case Theorem \ref{thm:tube} follows from the near-independence of the histories $\Xi_i$ and from the arguments in Sections \ref{sec:subtube} and \ref{sec:suptube}.



\end{document}